\documentclass[12pt]{article}%
\usepackage{graphicx}
\usepackage{mathrsfs}
\usepackage{amsmath}
\usepackage{amsfonts}
\usepackage[T2A]{fontenc}
\usepackage[cp1251]{inputenc}
\usepackage[russian,english]{babel}
\usepackage{amssymb}%

\usepackage{wasysym}
\newtheorem{theorem}{Theorem}

\newtheorem{corollary}[theorem]{Corollary}

\newtheorem{example}[theorem]{Example}

\newtheorem{lemma}[theorem]{Lemma}

\newtheorem{proposition}[theorem]{Proposition}
\newtheorem{remark}[theorem]{Remark}

\newenvironment{proof}[1][Proof]{\textbf{#1.} }{\ \rule{0.5em}{0.5em}}
\usepackage{setspace}
\begin{document}
\author {Martial Longla\\ University of Mississippi\\Department of mathematics}
\title{Remarks on limit theorems for reversible Markov processes and their applications} \date{January 2017}
\maketitle

\abstract{ We propose some backward-forward martingale decompositions for functions of reversible Markov chains. These decompositions are used to prove the functional CLT for reversible Markov chains with asymptotically linear variance of partial sums. We also provide a proof of the equivalence between asymptotic linearity of the variance and convergence of the integral of $1/(1-t)$ with respect to the associated spectral measure $\rho$. We also study the asymptotic behavior of linear processes having as innovations mean zero square integrable functions of stationary reversible Markov chains. We apply this study to several cases of reversible stationary Markov chains that arise in regression estimation. 
}



\section{Introduction}

An important theoretical question with numerous practical implications is to prove stability of the central limit theorem under formation of linear sums.
By this we understand that if $\displaystyle S_n(\xi)/\sqrt{n}=\sum_{i=1}^{n}\xi_{i}/\sqrt{n}$ converges in distribution to a normal variable, then the same result holds for $S_{n}(X)\ $ properly normalized, where $(X_i, 1\le i\le n) $ are linear functions of $(\xi_i, -\infty< i< \infty)$. This problem was first studied in the literature by Ibragimov (1962), who proved that if $(\xi_{i}, {i\in{\mathbb{Z}}})$ are centered i.i.d. with finite second moments, $S_{n}(X)/b_{n}$ satisfies the central limit theorem (CLT). The extra condition of finite second moment was removed by Peligrad and Sang (2013). The central limit theorem for $S_{n}(X)/b_{n}$ for the case when the innovations are square integrable martingale differences was proved by Peligrad and Utev (1997) and (2006), where an extension to generalized martingales was also given.

On the other hand, motivated by applications to unit root testing and to isotonic regression, a related question is to study the limiting behavior of $S_{[nt]}(X)/b_{n}$ (here and throughout the paper $[x]$ denotes the integer part of $x$). The first results on this topic are due to Davydov (1970), who established convergence to fractional Brownian motion for the case of i.i.d. innovations $(\xi_i, 1\le i\le n)$. Extensions to dependent settings under certain projection criteria can be found for instance in Wu and Min (2005) and Dedecker et al. (2011), among others.

Kipnis and Varadhan (1986) considered partial sums $S_{n}(X)$  of an additive mean zero functional of a stationary reversible ergodic Markov chain and showed that the convergence of $var(S_{n})/n(X)$ implies convergence of $\{S_{[nt]}(X)/\sqrt{n},$ $0\leq t\leq1\}$ to the Brownian motion. There is a considerable number of papers that further extend and apply this result to infinite particle systems, random walks, processes in random media, Metropolis-Hastings algorithms. Among others, Kipnis and Landim (1999) considered interacting particle systems, Tierney (1994) discussed the applications to Markov Chain Monte Carlo and Wu (1999) studied the law of the iterated logarithm.
Here, we will consider other cases of linear processes such as the causal model, applications to kernel estimation and linear regression.

We review the central limit theorem for stationary Markov chains with self-adjoint operator and general state space. We investigate the case when the variance of the partial sum is asymptotically linear in $n,$ and propose a new proof of the functional CLT for ergodic reversible Markov chains in Corollary 1.5 of Kipnis and Varadan (1986). We prove the equivalence of $\displaystyle \lim_{n\to\infty}var(S_{n}(\xi))/n<\infty $ and convergence of $\displaystyle \int_{-1}^{1}\frac{\rho(dt)}{1-t}$ for a mean zero function $f$ of a stationary reversible Markov chain. Here, $\rho$ is the spectral measure corresponding to $f$. This equivalence is used to provide a new forward-backward martingale decomposition for the given class of processes. Among new results of this paper, is a forward-backward martingale decomposition for stationary reversible Markov chains. In Proposition \ref{convL2}, we state a convergence theorem that helps establish a martingale convergence theorem in Lemma \ref{MartingaleConv}. A new proof of the central limit theorem based on Heyde (1974) is provided. Throughout this paper we use the spectral theory of bounded self-adjoint operators. In Section 1 we have the introduction, Section 2 is about the forward-backward martingale decomposition and Section 3 tackles the new proof of the functional central limit theorem for ergodic reversible Markov chains and Section 4 provides applications to various statistical models.

\subsection{Definitions and notations}

We assume that $(\gamma_{n})_{n\in\mathbb{Z}}$ is a stationary reversible Markov chain
defined on a probability space $(\Omega,\mathcal{F},\mathbb{P})$ with values
in a general state space $(S,\mathcal{A})$. The marginal distribution is
denoted by $\pi(A)=\mathbb{P}(\gamma_{0}\in A)$. Assume that there is a
regular conditional distribution for $\gamma_{1}$ given $\gamma_{0}$ denoted
by $$Q(x,A)=\mathbb{P}(\gamma_{1}\in A|\,\gamma_{0}=x).$$ Let $Q$ also denote
the Markov operator {acting via $$(Qg)(x)=\int_{S}g(s)Q(x,ds).$$ Next, let
$\mathbb{L}_{0}^{2}(\pi)$ be the set of measurable functions on $S$ such that
$\int g^{2}d\pi<\infty$ and $\int gd\pi=0.$ If }${g,h}\in${$\mathbb{L}_{0}%
^{2}(\pi),$ the integral }$\int_{S}g(s)h(s)d\pi$ will sometimes be denoted by
$<g,h>$.

{For some function }${g}\in${$\mathbb{L}_{0}^{2}(\pi)$, let }%
\begin{equation}
{\xi_{i}=g(\gamma}_{i}{),\ S_{n}(\xi)=\sum\limits_{i=1}^{n}\xi_{i},\ }%
\sigma_{n}({g)}=(\mathbb{E}S_{n}^{2}({\xi)})^{1/2}. \label{defcsi}%
\end{equation}
{\ Denote by $\mathcal{F}_{k}$ the $\sigma$--field generated by $\gamma_{i}$
with $i\leq k$ and by }$\mathcal{I}$ the invariant $\sigma-$field{. }

For any integrable random variable $X$ we denote $\mathbb{E}_{k}%
X=\mathbb{E}(X|\mathcal{F}_{k}).$ With this notation, $\mathbb{E}_{0}\xi
_{1}=Qg(${$\gamma$}$_{0})=\mathbb{E}(\xi_{1}|${$\gamma$}$_{0}).$ We denote by
${{||X||}_{p}}$ the norm in {$\mathbb{L}^{p}$}$(\Omega,\mathcal{F}%
,\mathbb{P}).$

The Markov chain is called reversible if $Q=Q^{\ast},$ where $Q^{\ast}$ is the
adjoint operator of $Q$. In this setting, the condition of reversibility is
equivalent to requiring that $(${$\gamma$}$_{0},${$\gamma$}$_{1})$ and
$(\gamma_{1},\gamma_{0})$ have the same distribution. Equivalently%
\[
\int_{A}Q(\omega,B)\pi(d\omega)=\int_{B}Q(\omega,A)\pi(d\omega)
\]
for all Borel sets $A,B\in\mathcal{A}$. The spectral measure of $Q$ with
respect to ${g}$ is concentrated on $[-1,1]$ and will be denoted by $\rho
_{g}.$ Then
\[
\mathbb{E}(Q^{m}g(\gamma_{0})Q^{n}g(\gamma_{0}))=<Q^{m}g,Q^{n}g>=\int
\nolimits_{-1}^{1}t^{n+m}\rho_{g}(dt).
\]

We denote by $W(t)$ is the standard Brownian motion.
All throughout the paper $\Longrightarrow$ denotes convergence in distribution, $\to^P$ denotes convergence in probability and $[x]$ is the integer part of $x$. 

We also need to introduce here some very useful notions from the spectral theory.

\subsection{Spectral Theory of self-adjoint operators}

Self-adjoint operators have spectral families with certain regularity properties, beyond the properties shared by all spectral families, which are very important in the proof of the theorems in this paper. Recall that a linear vector space $\mathbb{H}$ is a Hilbert space, if it is endowed with an inner product $<.,.>$, associated with a norm $||.||$ and metric $d(.,.)$, such that every Cauchy sequence has a limit in $\mathbb{H}$. Elements $x, y$ of a Hilbert space are said to be orthogonal if $<x,y> = 0$.
Suppose there is a non-decreasing family $(M (\lambda), \lambda\in\mathbb{R})$ of closed subspaces of $\mathbb{H}$ depending on a real parameter $\lambda$, such that the intersection of all the $M (\lambda)$ is $\{0\}$ and their union is dense in $\mathbb{H}$. Recall that The family is \textquotedblleft non-decreasing\textquotedblright\, if $M(\lambda_{1}) \subset M (\lambda_{2})$ for $\lambda_{1}< \lambda_{2}$.
This property also extends to the associated family $(E(\lambda), \lambda\in\mathbb{R})$ of orthogonal projections on $M(\lambda)$. The associated family of orthogonal projections is called spectral family or resolution of the identity if 
$\displaystyle\lim_{\lambda\to -\infty}E(\lambda)=0 \quad \mbox{and}\quad \lim_{\lambda\to \infty}E(\lambda)=1.$ 

{\bf Spectral theorem for self-adjoint operators in Hilbert spaces:}{\it

Every self-adjoint operator $Q$ in a Hilbert space $\mathbb{H}$ admits an expression $\displaystyle Q=\int_{-\infty}^{\infty}\lambda dE(\lambda)$ by means of a spectral family ($E(\lambda), \lambda\in\mathbb{R}$) which is uniquely determined by $Q$.}

The family ($E(\lambda), \lambda\in\mathbb{R}$) yields valuable information on the spectral structure of $Q$: the location of its singular or absolutely continuous spectrum and its eigenvalues. Also, it naturally leads to the definition of functions $f(Q)$, for a wide family of functions $f$. When the operator is bounded, the integral can be taken over the spectrum $\sigma(Q)$ of the operator (set of points $\lambda$ for which there is no bounded inverse to $Q-\lambda I$, where $I$ is the identity operator). This applies to Markov operators (a Markov operator is a unity-preserving positive contraction).
The inner product in a Hilbert space allows to define $Q^{*}$, the adjoint operator to $Q$ by the formula $\displaystyle <Qx,y>=<x,Q^{*}y>, \quad \forall x,y \in \mathbb{H}.$ The operator $Q$ is self-adjoint if the above yields $Q^{*}=Q$. For more, see Conway (1990).

{\bf Example of Markov operator:}
{\it Assume that $(\xi_{n}, n\in\mathbb{Z})$ is the Markov chain
defined above. $Q$ induces an
operator acting via $\displaystyle (Qf)(x)=\int_{S}f(s)Q(x,ds)$ in the Hilbert space $\mathbb{L}^{2}(\pi)$. The defined operator Q is a Markov operator with spectrum on $[-1,1]$.
For such $Q$, the above representation becomes:
$\displaystyle Q=\int_{-1}^{1}\lambda dE(\lambda),\quad \mbox{leading to}$
$\displaystyle <Qf,f>=\int_{-1}^{1}\lambda d<E(\lambda)f,f>=\int_{-1}^{1}\lambda d\rho (\lambda),$
where $\rho (\lambda)$ denotes the spectral measure of the operator applied to $f$. 
}

Based on this example, for a reversible Markov chain generated by $Q$,
$$
\mathbb{E}(\mathbb{E}(X_{k}|\mathscr{F}_{0})\mathbb{E}(X_{j}|\mathscr{F}_{0}))=\int_{-1}^{1}t^{k+j}\rho(dt).\label{SpQ}
$$

\section{More about reversible Markov chains}

\subsection{Another look at the variance} The following important by itself lemma holds.
\begin{lemma} \label{Kipeq}
\quad

Let $(\xi_{i}, i\in \mathbb{N})$ be defined by (\ref{defcsi}). Assume that $\rho$ is the spectral measure such that there is no atoms at $1$ and $-1$. Then, $\displaystyle
\frac{var(S_{n})}{n}\to \sigma^{2}<\infty \Longleftrightarrow \int_{-1}^{1}\frac{1}{1-t}\rho(dt)<\infty. \label{eqvar}
$

Moreover, $\displaystyle\lim_{n\to\infty}\frac{var(S_{n})}{n}=\int_{-1}^{1}\frac{1+t}{1-t}\rho(dt).$
\end{lemma}
\begin{proof}
It is well known that for a stationary reversible mean zero Markov chain $(\xi_{i}, i\in \mathbb{N})$,
$$\frac{var(S_{n})}{n}=\frac{1}{n}\Big(n~ var(X_0)+2\sum_{k=1}^{n-1}\sum_{j=k+1}^{n}\mathbb{E}(X_j X_k)\Big)=\frac{1}{n}\Big(n~ var(X_0)+2\sum_{k=1}^{n-1}\sum_{j=k+1}^{n}\mathbb{E}(X_0 X_{j-k})\Big).$$
The second equality sues stationarity of the Markov chain. Now, using the fact that for this sequence, we have the spectral representation
$$\displaystyle \mathbb{E}(X_{0}X_{k})=\mathbb{E}(\mathbb{E}(X_{0}|\mathscr{F}_{0})\mathbb{E}(X_{k}|\mathscr{F}_{0}))=\int_{-1}^{1}t^{k}\rho(dt), \quad \mbox{we obtain} $$

$$ \frac{var(S_{n})}{n}=\frac{1}{n}\Big(n~ \int_{-1}^{1}\rho(dt)+2\sum_{k=1}^{n-1}\sum_{j=k+1}^{n}\int_{-1}^{1}t^{j-k}\rho(dt)\Big)=\int_{-1}^{1}\rho(dt)+\frac{2}{n}\sum_{k=1}^{n-1}\sum_{u=1}^{n-k}\int_{-1}^{1}t^{u}\rho(dt).$$

Therefore, by simple calculations, we obtain 

$$\displaystyle\frac{var(S_{n})}{n}=\int_{-1}^{1}\frac{1+t}{1-t}\rho(dt)-\frac{2}{n}\sum_{k=1}^{n-1}\int_{-1}^{1}\frac{t^{n-k+1}}{1-t}\rho(dt)=\int_{-1}^{1}\frac{1+t}{1-t}\rho(dt)-\frac{2}{n}\sum_{j=2}^{n}\int_{-1}^{1}\frac{t^{j}}{1-t}\rho(dt).$$
Thus, if $\displaystyle\int_{-1}^{1}\frac{1}{1-t}\rho(dt)<\infty$, then we apply the Lebesgue's dominated convergence theorem to show that $\displaystyle \lim_{n\to\infty}\int_{-1}^{1}f_n(t)dt\to 0$, where $\displaystyle f_{n}(t)=
\frac{2t^{n}}{1-t}.$ The standard theorem on Cesaro means leads to 
$$\lim_{n\to\infty}\frac{2}{n}\sum_{j=2}^{n}\int_{-1}^{1}\frac{t^{j}}{1-t}\rho(dt)=0.$$
(For this sequence, we have $$|f_{n}|\leq 1/(1-t) \quad \rho-\mbox{almost everywhere,}$$ and $$f_{n}\to 0 \quad \rho-\mbox{almost everywhere}.$$ 
Moreover, $(1+t)/(1-t)$ is $\rho$-integrable, provided that $1/(1-t)$ is $\rho$-integrable. Therefore, $$\displaystyle\lim_{n\to\infty}\frac{var(S_{n})}{n}=\int_{-1}^{1}\frac{1+t}{1-t}\rho(dt) <\infty. $$

Applying Fatou lemma to the above functions, $$\displaystyle\lim_{n\to\infty}\frac{var(S_{n})}{n}\ge \int_{-1}^{1}\liminf_{n}(\frac{1+t}{1-t}-\frac{2}{n}\sum_{j=2}^{n}\frac{t^j}{1-t})\rho(dt)=\int_{-1}^{1}\frac{1+t}{1-t}\rho(dt).$$
Therefore, if $\displaystyle\frac{Var(S_{n})}{n}$ is convergent, then $\displaystyle\int_{-1}^{1}\frac{1+t}{1-t}\rho(dt)<\infty$. This leads to $\displaystyle \int_{-1}^{1}\displaystyle\frac{\rho(dt)}{1-t}<\infty$. So, $$\displaystyle\lim_{n\to\infty}\frac{var(S_{n})}{n}=\int_{-1}^{1}\frac{1+t}{1-t}\rho(dt)=2\int_{-1}^{1}\frac{1}{1-t}\rho(dt)-\mathbb{E}(X_{0}^{2}).$$
This leads to the conclusion of the lemma.
\end{proof}
\bigskip

\subsection{Forward-Backward martingale decomposition}

Martingale decomposition of sequences of random variables is a very important tool in probability theory. The proof of central limit theorems is often based on these decompositions. One shows that the variable can be represented as a sum of a martingale and a \textquotedblleft remainder\textquotedblright\ with suitable properties. For more on this topic, see Wu (1999), Wu and Woodroofe (2004), Zhao and al. (2010). For stationary reversible Markov chains, we obtain more flexibility to form martingale differences for triangular arrays. This allows to obtain in the limit (convergence in $\mathbb{L}^{2}$) martingale differences that sum up to martingales.

From Longla and al. (2012), for triangular arrays of random variables, we have 

$\displaystyle \xi_{k}+\xi_{k+1}=D_{k+1}^{n}+\tilde{D}_{k}^{n}+\frac{1}{n}\mathbb{E}_{k}
(S_{n}-S_{k})+\frac{1}{n}\mathbb{E}_{k+1}(S_{k+n+1}-S_{k+1}),$ where $\tilde{D}_{k}^{n}$ is the equivalent of $D_{k}^{n}$ for the reversed martingale, and
\begin{equation}
D_{k}^{n}=\frac{1}{n}\sum_{i=0}^{n-1}[\mathbb{E}_{k}(S_{k+i})-\mathbb{E}_{k-1}(S_{k+i})]\text{.} 
\label{martdef}%
\end{equation}
Denoting $B_{n,k}=\frac{1}{n}\mathbb{E}_{k}
(S_{n}-S_{k})$, from the above formula we obtain 
\begin{equation}
\xi_{k}+\xi_{k+1}=D_{k+1}^{n}+\tilde{D}_{k}^{n}+B_{n,k}+B_{n,k+1}.\label{Xk1}
\end{equation}

We shall show that $B_{n,k}$ and $B_{n,k+1}$ converge to $0$ in $\mathbb{L}^{2}$.

\begin{proposition}
Under the assumption of asymptotic linearity of the variance of partial sums, $B_{n,k}\to 0$ in $\mathbb{L}^{2}$ uniformly in $k$ as $n\to \infty$.
\end{proposition}

\begin{proof}
To show that $B_{n,k}$ converges uniformly in $k$ to $0$ in $\mathbb{L}^{2}$, it is enough to show that the variance converges to $0$ uniformly in $k$, and the expected value is equal to $0$. The mean zero assumption solves the problem of the expected value, and we have
$$\displaystyle var(B_{n,k})=\frac{1}{n^{2}}\mathbb{E}(\mathbb{E}_{k}(S_{n}-S_{k}))^{2}.$$
From stationarity, we obtain $$\displaystyle var(B_{n,k})=\frac{1}{n^{2}}\mathbb{E}(\mathbb{E}_{0}(S_{n-k}))^{2}.$$
Using the spectral theorem, we get 
$$\displaystyle var(B_{n,k})=\frac{1}{n^{2}}\int_{-1}^{1}(t+\cdots+t^{n-k})^{2}\rho{dt}=\int_{-1}^{1}f_{n}(t)\rho{dt}.$$
Here $$\displaystyle 0\leq f_{n}(t)=\frac{1}{n^{2}}(t+\cdots+t^{n-k})^{2}\leq \frac{2}{n}\frac{1}{1-t} \quad \rho-\mbox{almost surely.}$$
Applying Lemma \ref{Kipeq}, 
$$\displaystyle var(B_{n.k})\leq \frac{2}{n}\int_{-1}^{1}\frac{\rho(dt)}{1-t} \to 0 \quad \mbox{as}\quad n\to\infty \quad \mbox{because}\quad \int_{-1}^{1}\frac{\rho(dt)}{1-t}<\infty.$$
So, $B_{n,k}\to 0$ uniformly in $\mathbb{L}^{2}$. 
\end{proof}

\begin{proposition}[An $\mathbb{L}^{2}$ convergence theorem] \label{convL2}
\quad

Let $(\gamma_{i}, i\in\mathbb{Z})$ be a reversible stationary Markov chain. Let $(\xi_{i}, i\in\mathbb{Z})$ be defined by (\ref{defcsi}). If $\displaystyle
var(S_{n})/n \to \sigma^{2}\neq0, \quad\mbox{then}
$
$\displaystyle
\sum_{i=0}^{n}(\mathbb{E}(\xi_{i}|\mathscr{F}_{1})-\mathbb{E}(\xi_{i}|\mathscr{F}_{0})) \quad \mbox{converges in}\quad \mathbb{L}^{2}.
$ 
\end{proposition}
\bigskip

\begin{proof}
To prove Proposition \ref{convL2}, we shall show that the sequence is a Cauchy sequence in $\mathbb{L}_{2}$. Define
$$\displaystyle A_{n,p}=\mathbb{E}(\sum_{i=n}^{p}\mathbb{E}(\xi_{i}|\mathscr{F}_{1})-\mathbb{E}(\xi_{i}|\mathscr{F}_{0}))^{2}=\mathbb{E}(\mathbb{E}(S_{p}-S_{n-1}|\mathscr{F}_{1})-\mathbb{E}(S_{p}-S_{n-1}|\mathscr{F}_{0}))^{2},\quad
\forall p>n: p,n\in \mathbb{N}.$$
Squaring the quantity and computing the expected value by conditioning on $\mathscr{F}_{0}$ for the cross therm, taking into account the Markov property and the fact that $\mathscr{F}_{0} \subset \mathscr{F}_{1}$, we obtain
$$\displaystyle A_{n,p}=\mathbb{E}(\mathbb{E}(S_{p}-S_{n-1}|\mathscr{F}_{1}))^{2}-\mathbb{E}(\mathbb{E}(S_{p}-S_{n-1}|\mathscr{F}_{0}))^{2}=$$
$$\displaystyle \mathbb{E}(\mathbb{E}(S_{p-1}-S_{n-2}|\mathscr{F}_{0}))^{2}-\mathbb{E}(\mathbb{E}(S_{p}-S_{n-1}|\mathscr{F}_{0}))^{2}.$$
Recalling from spectral calculus that for a reversible Markov chain we have the representation (\ref{SpQ}), we obtain
$$\displaystyle A_{n,p}=\int_{-1}^{1}(\sum_{j=n-1}^{p-1}t^{j})^{2}\rho(dt)-\int_{-1}^{1}(\sum_{j=n}^{p}t^{j})^{2}\rho(dt)=\int_{-1}^{1}(\sum_{j=n-1}^{p-1}t^{j})^{2}(1-t^{2})\rho(dt).$$

\noindent $\displaystyle A_{n,p}=\int_{-1}^{1}t^{2n-2}(1+t+\cdots+t^{p-n})^{2}(1-t^{2})\rho(dt)=\int_{-1}^{1}t^{2n-2}\frac{(1-t^{p-n+1})^{2}(1-t^{2})}{(1-t)^{2}}\rho(dt),$
leading to
$$\displaystyle A_{n,p}=\int_{-1}^{1}\frac{t^{2n-2}(1-t^{p-n+1})^{2}(1+t)}{1-t}\rho(dt)\leq 8\int_{-1}^{1}\frac{t^{2n-2}}{1-t}\rho(dt).$$
By the Lebesgue dominated convergence theorem, $$\displaystyle \limsup_{n}\int_{-1}^{1}\frac{t^{2n-2}}{1-t}\rho(dt)\leq \int_{-1}^{1}\limsup_{n}\frac{t^{2n-2}}{1-t}\rho(dt)=0.$$
Therefore, $A_{n,p}\to 0$ as $n \to \infty$. So, the sequence is a Cauchy sequence in $\mathbb{L}^{2}$.
\end{proof}

We shall now show the following lemma to finish the martingale decomposition.

\begin{lemma}[A martingale difference convergence theorem] \label{MartingaleConv}
\quad

If $Var(S_{n})/n\to \sigma^{2}<\infty$, then the sequences $D_{k}^{n}$ and $\tilde{D}_{k}^{n}$ defined above converge in $\mathbb{L}^{2}$ respectively to a martingale difference sequence and a reversed martingale difference.
\end{lemma}

\begin{proof}
The proof of Lemma \ref{MartingaleConv} is based on Proposition \ref{convL2}. We have, by Proposition \ref{convL2}, 
$$\displaystyle
\mathbb{E}(S_{n}|\mathscr{F}_{1})-\mathbb{E}(S_{n}|\mathscr{F}_{0})=\sum_{i=0}^{n}(\mathbb{E}(\xi_{i}|\mathscr{F}_{1})-\mathbb{E}(\xi_{i}|\mathscr{F}_{0})) \quad \mbox{converges in}\quad \mathbb{L}^{2}.
$$ 
Let $D_{1}$ be the limit in $\mathbb{L}^{2}$ of $\mathbb{E}(S_{n}|\mathscr{F}_{1})-\mathbb{E}(S_{n}|\mathscr{F}_{0})$. By the standard theorem on Cesaro means, we have
$$\displaystyle D_{1}^{n}=\frac{1}{n}\sum_{i=1}^{n-1}[\mathbb{E}(S_{i}|\mathscr{F}_{1})-\mathbb{E}(S_{i}|\mathscr{F}_{0})] \to D_{1} \quad\mbox{in} \quad \mathbb{L}^{2}.$$ The formula of $D_{k}^{n}$ can be obtained from Section 3.1 of Longla and al.
Similarly, we obtain $$D^{n}_{k}\to D_{k} \quad \mbox{in}\quad \mathbb{L}^{2} \quad \mbox{and} \quad D_{k} \quad \mbox{is}\quad \mathscr{F}_{k}-\mbox{measurable}.$$ 
By Jensen's inequality, and the double expectation rule,
$$\displaystyle\mathbb{E}\big(\mathbb{E}(D_{k}^{n}-D_{k}|\mathscr{F}_{k-1})\big)^{2}\leq \mathbb{E}\big(\mathbb{E}((D_{k}^{n}-D_{k})^{2}|\mathscr{F}_{k-1})\big)=\mathbb{E}(D_{k}^{n}-D_{k})^{2}.$$
Due to convergence in $\mathbb{L}^{2}$ of $D^{n}_{k}$ to $D_{k}$, it follows that $$\mathbb{E}(D_{k}^{n}|\mathscr{F}_{k-1}) \quad \mbox{converges in}\quad \mathbb{L}^{2} \quad \mbox{to}\quad \mathbb{E}(D_{k}|\mathscr{F}_{k-1}).$$ Thus, $\mathbb{E}(D_{k}|\mathscr{F}_{k-1})=0$ almost surely, because $\mathbb{E}(D^{n}_{k}|\mathscr{F}_{k-1})=0$. So, $(D_{k}, \mathscr{F}_{k}, k\in\mathbb{N})$ is a direct martingale difference. Due to stationarity of the initial sequence, $D_{k}^{n}$ is a stationary sequence. So, $D_{k}$ is stationary. The proof of the second part of the lemma is similar.
\end{proof}

\begin{proposition}[Forward-backward martingale decomposition]\label{FBMD}
\quad 

Let $(\xi_{i}, i\in \mathbb{Z})$ be defined by formula (\ref{defcsi}). Let $Var(S_{n})/n\sigma^{2}<\infty$. Then, $\displaystyle 2S_{n}=M_{n}^{d}+M_{n}^{r}+\xi_{n}-\xi_{0},$
where $M_{n}^{d}$, $M_{n}^{r}$ are direct and reversed martingales respectively.
\end{proposition}

\begin{proof}
Recalling that $B_{n,k}\to 0$ , $B_{n,k+1}\to 0$ in $\mathbb{L}^{2}$, using Lemma \ref{MartingaleConv} and the representation of $\xi_{k}+\xi_{k+1}$ by formula (\ref{Xk1}), we obtain as $n\to \infty$, $\xi_{k}+\xi_{k+1}=D_{k+1}+\tilde{D}_{k}$. It follows that
\begin{equation}
2S_{n}=\sum_{i=1}^{n}D_{i}+\sum_{i=1}^{n}\tilde{D}_{i}+\xi_{n}-\xi_{0}.
\label{Martdec}
\end{equation}
\end{proof}

\subsection{Central limit theorem}

\begin{theorem}\label{HRevCLT} 
\quad

Let $(\gamma_{i}, i\in\mathbb{N})$ be a reversible ergodic Markov chain. If a mean zero sequence is defined by (\ref{defcsi}) with $\mathbb{E}\xi_{0}^{2} <\infty$, and $var(S_{n})/n\to \sigma^{2}\neq0$, then $S_{n}/\sigma\sqrt{n}\Longrightarrow N(0,1)$.
\end{theorem}

Theorem \ref{HRevCLT} was stated and proved by Kipnis and Varadhan (1986). Here, we provide a different proof of this theorem based on the following result of Heyde (1974):

\begin{theorem}[Heyde]\label{Hey}
\quad

Let $(\xi_{i}, i\in\mathbb{Z})$ be a stationary and ergodic mean zero sequence of random variables with finite second moments defined by (\ref{defcsi}). Assume that the following two conditions hold:
\begin{eqnarray}
\sum_{i=0}^{n}(\mathbb{E}(\xi_{i}|\mathscr{F}_{1})-\mathbb{E}(\xi_{i}|\mathscr{F}_{0})) \quad \mbox{converges in}\quad \mathbb{L}^{2}, \label{cond1}\\
var(S_{n})/n \to \sigma^{2}= \mathbb{E}(\sum_{i=1}^{\infty}(\mathbb{E}(\xi_{i}|\mathscr{F}_{1})-\mathbb{E}(\xi_{i}|\mathscr{F}_{0})))^2, \label{cond2}
\end{eqnarray} 
where $\mathscr{F}_{i}$ is the $\sigma$-field generated by $(\gamma_{j}, j\leq i)$. Then, $\displaystyle n^{-1/2}S_{n} \Longrightarrow N(0, \sigma^{2})$.
\end{theorem}
\bigskip

\begin{proof}
To prove Theorem \ref{HRevCLT}, we shall verify the assumptions of Theorem \ref{Hey}. The assumption (\ref{cond2}) is partially common to both theorems. Notice that 

$$\displaystyle \mathbb{E}(\sum_{i=1}^{n}(\mathbb{E}(\xi_{i}|\mathscr{F}_{1})-\mathbb{E}(\xi_{i}|\mathscr{F}_{0})))^2=\sum_{j=1}^{n}\sum_{i=1}^{n}\mathbb{E}\Big((\mathbb{E}(\xi_{i}|\mathscr{F}_{1})-\mathbb{E}(\xi_{i}|\mathscr{F}_{0}))(\mathbb{E}(\xi_{j}|\mathscr{F}_{1})-\mathbb{E}(\xi_{j}|\mathscr{F}_{0}))\Big)= $$

$$=\displaystyle \sum_{j=1}^{n}\sum_{i=1}^{n}\Big( \mathbb{E}(\mathbb{E}(\xi_{i}|\mathscr{F}_{1})\mathbb{E}(\xi_{j}|\mathscr{F}_{1}))-\mathbb{E}(\mathbb{E}(\xi_{i}|\mathscr{F}_{0})\mathbb{E}(\xi_{j}|\mathscr{F}_{0}))\Big)=A.$$ This equality uses the fact that $\mathscr{F}_{0}\subset\mathscr{F}_{1}$ leads to $\displaystyle \mathbb{E}(\mathbb{E}(\xi_{i}|\mathscr{F}_{0})\mathbb{E}(\xi_{j}|\mathscr{F}_{1}))=\mathbb{E}(\mathbb{E}(\xi_{i}|\mathscr{F}_{0})\mathbb{E}(\xi_{j}|\mathscr{F}_{0}))$. 
Using stationarity and reversibility, we obtain 
$$\displaystyle A =\displaystyle \sum_{j=1}^{n}\sum_{i=1}^{n}\Big( \mathbb{E}(\mathbb{E}(\xi_{i-1}|\mathscr{F}_{0})\mathbb{E}(\xi_{j-1}|\mathscr{F}_{0}))-\mathbb{E}(\mathbb{E}(\xi_{i}|\mathscr{F}_{0})\mathbb{E}(\xi_{j}|\mathscr{F}_{0}))\Big)=$$

$$\displaystyle =\sum_{j=1}^{n}\sum_{i=1}^{n}\int_{-1}^{1}(t^{i+j-2}-t^{i+j})\rho(dt)=\int_{-1}^{1}\frac{(1+t)(1-t^n)^2 }{1-t}\rho(dt).$$ Thus, by Lemma \ref{Kipeq}, $$\displaystyle \mathbb{E}(\sum_{i=1}^{\infty}(\mathbb{E}(\xi_{i}|\mathscr{F}_{1})-\mathbb{E}(\xi_{i}|\mathscr{F}_{0})))^2=\int_{-1}^{1}\frac{1+t}{1-t}\rho(dt)=\lim n^{-1}var(S_n ).$$ 

Assumption (\ref{cond1}) follows from Proposition \ref{convL2}. Thus, the conclusion holds.
\end{proof}

Note that the only assumption of reversibility in Theorem \ref{HRevCLT} drops all assumptions on mixing rates that are usually imposed on the Markov chain. Proposition 1 of Dedecker and Rio (2000), reformulated for stationary martingales reads as follows.

\begin{proposition} \label{Dedthight}
\quad 

Let $(D_{i}, i\in \mathbb{Z})$ be a stationary sequence of martingale differences or reversed martingale differences. Let $S_{n}$ be the partial sums of any of the sequences. Let $\lambda$ be a nonnegative real number and $\Gamma_{k}=(S^{*}_{k}>\lambda)$, where $\displaystyle S^{*}_{k}=\max_{1\leq i\leq k}(0, S_{1},\cdots,S_{k})$. Then,
\begin{equation}
\mathbb{E}((S_{n}^{*}-\lambda)^{2}_{+})\leq 4\sum_{i=1}^{n}\mathbb{E}(D^{2}_{k}\mathbb{I}_{\Gamma_{k}}), \label{uint}
\end{equation}
and $\displaystyle n^{-1}\max_{1\leq i\leq n}S^{2}_{i} \quad \mbox{is uniformly integrable.}$
\end{proposition}
The proof of the first part of the conclusion of this proposition can be found in Dedecker and Rio (2000). The second part concerning uniform integrability follows from the inequality (\ref{uint}). 
Denoting $\displaystyle M^{*}_{n}=\max_{1\leq i\leq n}|S_{i}|$, from inequality (\ref{uint}) applied to ($D_{i}$) and ($-D_{i}$), we have
$$\displaystyle \displaystyle n^{-1}\mathbb{E}((M_{n}^{*}-\lambda)^{2}_{+})\leq 8n^{-1}\sum_{i=1}^{n}\mathbb{E}(D^{2}_{k}\mathbb{I}_{\Gamma_{k}}).$$
Using stationarity, we obtain $$ \displaystyle n^{-1}\mathbb{E}((M_{n}^{*}-\lambda)^{2}_{+})\leq 8n^{-1}\sum_{i=1}^{n}\mathbb{E}D^{2}_{k}=8\mathbb{E}D_{0}^{2}.$$ Thus, taking $\lambda=0,$ we get $\displaystyle n^{-1}\max_{1\leq i\leq n}S^{2}_{i} \quad \mbox{is uniformly bounded in}\quad \mathbb{L}^{1}.$ 

The proof of uniform integrability of $\displaystyle n^{-1}\max_{1\leq i\leq n}S^{2}_{i}$ reduces to showing that $$\displaystyle\limsup_{n}\int_{A}n^{-1}\max_{1\leq i\leq n}S^{2}_{i}dP \to 0 \quad \mbox{as}\quad \mathbb{P}(A)\to 0.$$
This convergence follows is proved below.
$$\displaystyle\int_{A}n^{-1}\max_{1\leq i\leq n}S^{2}_{i}dP\leq 2\int_{A}n^{-1}((\max_{1\leq i\leq n}|S_{i}|-\lambda\sqrt{n})^{2}+(\lambda\sqrt{n})^{2})dP=$$ $$=2\int_{A}n^{-1}(\max_{1\leq i\leq n}|S_{i}|-\lambda\sqrt{n})^{2}dP+2\lambda^{2}\mathbb{P}(A).$$
$$\mbox{So,}\quad \int_{A}n^{-1}\max_{1\leq i\leq n}S^{2}_{i}dP\leq 2n^{-1}\mathbb{E}(M_{n}^{*}-\lambda\sqrt{n})^{2}_{+}+2\lambda^{2}\mathbb{P}(A)\leq 16n^{-1}\sum_{i=1}^{n}\mathbb{E}(D^{2}_{k}\mathbb{I}_{(S_{k}>\lambda\sqrt{n})})+2\lambda^{2}\mathbb{P}(A).$$
Due to stationary of the martingale differences $(D_{i}, i\in \mathbb{N})$, the sequence $(D_{i}^{2}, 1\leq i\leq n)$ is uniformly integrable. Therefore, for all $\varepsilon$ and for all $k$, there exists $n$ such that $\mathbb{E}(D^{2}_{k}\mathbb{I}_{S_{k}>\lambda\sqrt{n}})<\varepsilon$.
Thus, 
$$\displaystyle\int_{A}n^{-1}\max_{1\leq i\leq n}S^{2}_{i}dP\leq 16\varepsilon+2\lambda^{2}\mathbb{P}(A).$$
Finally, we obtain $$\displaystyle\lim_{\mathbb{P}(A)\to 0}\limsup_{n}\int_{A}n^{-1}\max_{1\leq i\leq n}S^{2}_{i}dP\leq 16\varepsilon.$$ 
Taking $\varepsilon \to 0$ leads to the conclusion of the proposition. Similar calculations provide the proof for the case of the reversed martingale differences.

Now we are ready to propose a new proof of Corollary 1.5 of Kipnis and Varadhan (1986).

\begin{theorem}[Kipnis, Varadhan] \label{KV1}
\quad

For any reversible stationary Markov chain $(\gamma_{j}, j\in\mathbb{Z})$ defined on a space $\mathcal{X}$, and for any mean zero function $g$ satisfying the following conditions:
\begin{enumerate}
\item $\int g^{2}(x)\pi(dx)< \infty,$
\item $\displaystyle
\lim_{n\rightarrow \infty} \frac{1}{n}E(g(\gamma_{1}) + \cdots + g(\gamma_{n}))^{2} = \sigma_g^2 < \infty,$
\end{enumerate}
the reversible Markov chain defined by $\xi_i=g(\gamma_i)$ satisfies,
$\displaystyle
\frac{S_{[nt]}}{\sqrt{n}}\Longrightarrow|\sigma_{g}|W(t)\text{.} \label{IP}%
$
\end{theorem}

\begin{proof}
To prove Theorem \ref{KV1}, we need to show tightness of $W_{n}(t)=S_{[nt]}/\sqrt{n}$. This means, show that $$\forall \varepsilon>0,
\quad \displaystyle\lim_{\delta\to 0}\limsup_{n\to \infty}\mathbb{P}(\sup_{|s-t|<\delta}|W_{n}(t)-W_{n}(s)|>\varepsilon)=0.$$ Convergence of finite dimensional distributions repeats the steps of Theorem 1 of Longla and al. (2012).
By Billingsley's Theorem 8.3 (1968) formulated for random elements of D (see page 137 or formula (8.16) in Billingsley, 1968) , taking into account stationarity of the process, this condition is satisfied if 
$$\displaystyle \lim_{\delta\to 0}\limsup_{n\to \infty}\frac{1}{\delta}\mathbb{P}(\max_{1\leq j\leq [n\delta]}|S_{j}|>\varepsilon\sqrt{n})=0.$$
Therefore, it is enough to have 
\begin{equation}
\frac{1}{n}\max_{1\leq j\leq n}S_{j}^{2} - \mbox{is uniformly integrable.} \label{tightness}
\end{equation}
The chain being reversible, we have by Proposition \ref{FBMD}:
\begin{equation}
2S_{n}=M^{d}_{n}+M^{r}_{n}+\xi_{n}-\xi_{0}, \label{decom}
\end{equation}
where $M^{d}_{n}=\sum_{i=1}^{n}D^{d}_{i}$ and $M^{r}_{n}=\sum_{i=1}^{n}D^{r}_{i}$ are respectively a direct and a reversed martingales. The sequences ($D^{d}_{i}$) and ($D_{i}^{r}$) are respectively stationary martingale differences and stationary reversed martingale differences.
Due to the representation (\ref{decom}), the condition (\ref{tightness}) is satisfied if \begin{eqnarray}
\frac{1}{n}\max_{1\leq i\leq n}(M^{s}_{i})^{2} \quad \mbox{is uniformly integrable for}\quad s=d,r, \label{condtight}\\
\mbox{and}\quad \frac{1}{n}\max_{1\leq i\leq n}(\xi_{i})^{2} \quad \mbox{is uniformly integrable} \label{condunif}.
\end{eqnarray} 
The condition (\ref{condtight}) is satisfied due to Proposition \ref{Dedthight} and (\ref{condunif}) is satisfied due to stationarity of the process ($\xi_{i},i\in \mathbb{Z}$). This concludes the proof of the theorem.
\end{proof}

\section{Linear functions of reversible processes }

\subsection{Overview}

We are interested in estimating regression parameters or functions and provide confidence intervals for this estimation. Deriving central limit theorems in this case is crucial because the lack of it would prevent any progress in solving the problem at hands. 
Let $(\xi_{i}, i\in{\mathbb{Z}})$ be a stationary sequence of random
variables on a probability space $(\Omega,\mathcal{K},\mathbb{P})$ with finite
second moments and zero means $(\mathbb{E}\xi_{0}=0)$. Let $(a_{i}, i\in{\mathbb{Z}})$ be a sequence of real numbers such that $\sum
\nolimits_{i\in{\mathbb{Z}}}a_{i}^{2}<\infty$ and denote by%
\begin{align}
X_{k} & =\sum_{j=-\infty}^{\infty}a_{k+j}\xi_{j}\;,\;S_{n}(X)=S_{n}%
=\sum_{k=1}^{n}X_{k},\quad\label{defXn}\\
b_{n,j} & =a_{j+1}+\ldots+a_{j+n}\quad\mbox{and}\quad b_{n}^{2}%
=\sum_{j=-\infty}^{\infty}b_{n,j}^{2}.\nonumber
\end{align}
The linear process $(X_{k}), k\in{\mathbb{Z}})$ is widely used in a variety of applied fields. It is properly defined for any square summable sequence $(a_{i}, i\in{\mathbb{Z}})$ if and only if the stationary sequence of innovations $(\xi_{i}, i\in{\mathbb{Z}})$ has a bounded spectral density. In general, the covariances of $(X_{k}, k\in{\mathbb{Z}})$ might not be summable.
Thus, the linear process might exhibit long-range dependence.
Peligrad (2012) showed the following theorem: 

\begin{theorem}
\label{cltlingenr} Assume that $(\xi_{j}, j\in{\mathbb{Z}})$ is defined by
(\ref{defcsi}) and $Q=Q^{\ast}$. Define $(X_{k}, k\in\mathbb{N})$, $S_{n}$ and $b_{n}$ as in
(\ref{defXn}). Assume that $b_{n}\rightarrow\infty$ as $n\rightarrow\infty$
and
\begin{equation}
\sum_{j\geq0}|\mathrm{cov}(\xi_{0},\xi_{j})|<\infty. \label{abscov}%
\end{equation}
Then, there is a non-negative random variable $\eta$ measurable with respect to $\mathcal{I}$ such that $n^{-1}\mathbb{E}((\sum_{k=1}^{n}\xi_{k}%
)^{2}|\mathcal{F}_{0})\rightarrow\eta$\ in $\mathbb{L}_{1}$ as $n\rightarrow
\infty$ and $\mathbb{E}\eta=\sigma_{g}^{2}.$ In addition
$$
\lim_{n\rightarrow\infty}\frac{\mathrm{var}(S_{n}(X))}{b_{n}^{2}}=\sigma
_{g}^{2}%
$$
and
\begin{equation}
\frac{S_{n}(X)}{b_{n}}\Longrightarrow\sqrt{\eta\ }N\text{ as }n\rightarrow\infty,
\label{CLT}%
\end{equation}
where $N$ is a standard normal variable independent on $\eta.$ Moreover if the
sequence $(\xi_{i}, i\in{\mathbb{Z}})$ is ergodic the central limit theorem in (\ref{CLT}) holds with $\eta=\sigma_{g}^{2}.$
\end{theorem}

She also mentioned that under the conditions of this theorem $\sigma_{g}^{2}$
also has the following interpretation: the stationary sequence $(\xi
_{i}, i\in{\mathbb{Z}})$ has a continuous spectral density $f(x)$ and
$\sigma_{g}^{2}=2\pi f(0).$

\subsection{A CLT for linear functions of reversible Markov chains}
Let
\begin{equation}
\Gamma_{j}=\sum_{k=0}^{\infty}|\mathbb{E(}\xi_{j+k}\mathbb{E}_{0}\xi
_{j})|<\infty\text{ and}\;\frac{1}{p}\sum_{j=1}^{p}\Gamma_{j}\rightarrow
0\text{ as}\;p\rightarrow\infty, \label{Mgen}%
\end{equation}
\begin{equation}
\sum_{i\in\mathbb{Z}}d_{n,i}^{2}\rightarrow c^{2}\,\text{and}\,\sum
_{j\in\mathbb{Z}}^{\ }(d_{n,j}-d_{n,j-1})^{2}\rightarrow0\,\text{as
}n\rightarrow\infty, \label{A}%
\end{equation}
and 
\begin{equation}
\sup_{j\in\mathbb{Z}}|d_{n,j}|\rightarrow0\,\ {\text{as}}\text{ \ }%
n\rightarrow\infty. \label{B}%
\end{equation} 
Define%
\begin{equation}
S_{n}=\sum_{i=1}^{n}d_{n,i}(\xi_{i}+\xi_{i+1}%
)\; \label{defX'}%
\end{equation}
Note that under assumptions of reversibility and $\displaystyle \frac{\sigma^2_n}{n} \to \sigma^2_{g}$, assumption (\ref{Mgen}) is satisfied. In fact, 

$$|\mathbb{E}(\xi_{j+k}\mathbb{E}_0 (\xi_j))|= |\int_{-1}^{1}t^{2j+k}\rho(dt)|\leq \int_{-1}^{1}|t|^{2j+k}\rho(dt).$$
Thus, $$\displaystyle \Gamma_j \leq \int_{-1}^{1}\sum_{k=0}^{\infty}t^{2j}|t|^{k}\rho(dt)=\int_{-1}^{1}\frac{t^{2j}}{1-|t|}\rho(dt)\leq \int_{-1}^{1}\frac{t^{2j}}{1-t}\rho(dt).$$
From this, it is clear that $\displaystyle \Gamma_j <\infty,$ $\displaystyle \lim_{j\to \infty}\Gamma_j =0$ and by the standard theorem on Cesaro means, 
$$\displaystyle \lim_{p\to \infty}\frac{1}{p}\sum_{j=1}^{p}\Gamma_j =0.$$
Therefore, we have the following:
\begin{theorem}
\label{Rev-CLT} Assume that $(\xi_{j}, j\in\mathbb{Z})$ defined by (\ref{defcsi}) is reversible, ergodic and $\displaystyle \frac{\sigma^2_n}{n} \to \sigma^2_{g}$.
Then, under assumptions (\ref{A}) and (\ref{B}), the CLT holds
in the forms $S_{n}\Longrightarrow N(0, \eta c^2)$ and $S_{[nt]}\Longrightarrow \sqrt{\eta}c W(t)$. In this case $\eta$ is defined as $n^{-1}\mathbb{E}(\sum_{k=1}^{n}(\xi_{k}+\xi_{k+1})^{2}|\mathcal{F}_{0})\rightarrow\eta$\ in $\mathbb{L}_{1}%
(\Omega,\mathcal{F},\mathbb{P})$ as $n\rightarrow\infty$. Furthermore, the stationary sequence $(\xi_{k}+\xi_{k+1})_{k\in{\mathbb{Z}}}$ has a continuous spectral density $h(x)$ and $\eta=2\pi h(0).$
\end{theorem}

A very important corollary to this theorem will follow from the fact, that 
$$S_{n}=\sum_{i=1}^{n}d_{n,i}(\xi_{i}+\xi_{i+1})=\sum_{i=1}^{n+1}\tilde{d}_{n,i}\xi_{i}, \quad \tilde{d}_{n,i}=d_{n,i}+d_{n,i-1}, \quad d_{n,0}=d_{n,n+1}=0.$$
Therefore, any sum of the form (\ref{defX'}) comes from a sum 
\begin{equation} \tilde{S}_{n}=\sum_{i=1}^{n}\tilde{d}_{n,i}\xi_{i}+d_{n,n}\xi_{n+1}, \quad d_{n,i}=\sum_{j=1}^{i}(-1)^{i+j}\tilde{d}_{n,j}. \label{DefS}
\end{equation}
Notice that for any stationary Markov chain and any sequence $d_{n,i}$ satisfying the conditions of the above theorem, we have $d_{n,n}\to 0$. So, we obtain the following.
\begin{theorem}
Under the conditions Theorem \ref{Rev-CLT}, with $d_{n,i}$ defined in (\ref{DefS}) , $\tilde{S}_{n} \Longrightarrow N(0, \eta c^{2})$ and $\tilde{S}_{[nt]} \Longrightarrow \sqrt{\eta}c W(t)$.
\end{theorem}

\section{Applications to statistical some models}
\subsection{The linear regression estimates}
Many statistical procedures, such as estimation of regression coefficients, produce linear statistics of type (\ref{defcsi}). For more information, see Chapter 9 in Beran (1994) for parametric regression or the paper by Robinson (1997) for nonparametric regression with use of kernel estimations.

\subsubsection{The linear regression problem without intercept}
We consider here the simple parametric regression model $Y_{i}=\beta X_{i}+\xi_{i}$, where the errors $\xi_{i}$ form a stationary reversible mean zero Markov chain, $X_{i}$ is a sequence of real-valued explanatory variables and $\beta$ is the parameter of interest. It is well-known that the least squares estimator of $\beta$ for a sample of size $n$ is $\hat{\beta}=\sum_{i=1}^{n}\alpha_{i}Y_{i}/\sum_{i=1}^{n}\alpha_{i}^{2}$. For this estimator, we have:
\begin{equation}
S_{n}:=\sqrt{\sum_{i=1}^{n}X_{i}^{2}}(\hat{\beta}-\beta)=\sum_{i=1}^{n}\tilde{d}_{n,i}\xi_{i},\quad \tilde{d}_{n,i}=\frac{X_{i}}{\sqrt{\sum_{i=1}^{n}X_{i}^{2}}}, \quad i \in \{1, \cdots, n\}
\label{defreg}
\end{equation}
On the other hand, the coefficients satisfy the following:
$$\sum_{i=1}^{n}d^{2}_{n,i}=\frac{\sum_{i=1}^{n}(\sum_{j=1}^{i}(-1)^{j}X_{j})^{2}}{\sum_{i=1}^{n}X_{i}^{2}}, \quad \sum_{i=1}^{n}(d_{n,i}-d_{n,i-1})^{2}=\frac{\sum_{i=1}^{n}(2\sum_{j=1}^{i-1}(-1)^{j}X_{j}+(-1)^{i}X_{i})^{2}}{\sum_{i=1}^{n}X_{i}^{2}}$$
$$\mbox{and} \quad |d_{n,n}| \le \max_{1\le i\le n}|d_{n,i}|=\frac{\max_{1\le i\le n}|\sum_{j=1}^{i}(-1)^{j}X_{j}|}{\sqrt{\sum_{i=1}^{n}X_{i}^{2}}}.$$
It follows that


\begin{theorem} \label{Linreg}
For the parametric linear regression problem above, if $\displaystyle \frac{\sigma^2_n}{n} \to \sigma^2_{g}$ and the following conditions are satisfied, 
\begin{enumerate}
\item $\xi_{i}$ is a reversible ergodic Markov chain,
\item $$\frac{\max_{1\le i\le n}|\sum_{j=1}^{i}(-1)^{j}X_{j}|}{\sqrt{\sum_{i=1}^{n}X_{i}^{2}}} \to 0 \quad \mbox{as}\quad n \to \infty,$$
\item $$\frac{\sum_{i=1}^{n}(2\sum_{j=1}^{i-1}(-1)^{j}X_{j}+(-1)^{i}X_{i})^{2}}{\sum_{i=1}^{n}X_{i}^{2}} \to 0 \quad \mbox{as}\quad n\to \infty,$$
\item $$\frac{\sum_{i=1}^{n}(\sum_{j=1}^{i}(-1)^{j}X_{j})^{2}}{\sum_{i=1}^{n}X_{i}^{2}} \to c^{2} \quad \mbox{as}\quad n\to \infty,$$
\end{enumerate}
then the CLT holds in the form 
\begin{equation}
\sqrt{\sum_{i=1}^{n}X_{i}^{2}}(\hat{\beta}-\beta) \Longrightarrow N(0, 2\pi h(0)c^{2}),
\end{equation}
where $h$ is the spectral density of the stationary sequence $\xi_{i}+\xi_{i+1}$.
\end{theorem}

\begin{example}
Take in the linear regression problem $X_2=\cdots=X_{n}=X$ and $X_{1}=X/2$, and a stationary ergodic error sequence. It follows that $\displaystyle\tilde{d}_{n,i}=\frac{2}{\sqrt{4n-3}}$, for $1<i\leq n$ and $\displaystyle\tilde{d}_{n,1}=d_{n,1}=\frac{1}{\sqrt{4n-3}}=d_{n,i}$ for all $i\leq n$.
So, $\displaystyle\lim_{n\to\infty}\sum_{i=1}^{n}d^{2}_{n,i}=\frac{1}{4}$, and all the assumptions of Theorem \ref{Linreg} are satisfied. Thus, 
\begin{equation}
X\sqrt{n}(\hat{\beta}-\beta) \Longrightarrow N(0, \pi h(0)/2),
\end{equation}
where $h$ is the spectral density of the stationary sequence $\xi_{i}+\xi_{i+1}$.
\end{example}

\begin{remark}
Notice that adding or subtracting a term that converges to zero does not influence the convergence. Thus, we can conclude that the result will hold for any $X=X_{1}=\cdots=X_{n}$. This is the case of the simple linear model problem $Y_{i}=\mu+\xi_{i}$, extended to the dependent case. 
It is easy to show that for an i.i.d sequence of innovations, $h(0)=4\sigma^2_g/2\pi=2\sigma^2_g/\pi$. Thus, the limiting variance is $\sigma_g^2$, giving the central limit theorem for i.i.d. sequences. 
\end{remark}

Notice that in the above example, we can also look for the central limit theorem in the form 
\begin{equation} b_n (\hat{\beta}-\beta)\to N(0, \sigma^2). \label{CLTbeta} \end{equation} 
For the usual central limit theorem, one can consider the formula in the special case when $b_n=\sqrt{n}.$ We have

$$b_n(\hat{\beta}-\beta)=\sum_{i=1}^{n}\frac{b_n X_i}{\sum_{j=1}^{n}X_j^2}\varepsilon_i.$$
Therefore, the following result follows.

\begin{corollary} \label{corbeta}
The central limit theorem \ref{CLTbeta} holds for the parametric linear regression problem above with $\sigma^2=2\pi h(0)c^2$, if $\displaystyle \frac{\sigma^2_n}{n} \to \sigma^2_{g}$ and the following conditions are satisfied, 

\begin{enumerate}
\item $\xi_{i}$ is a reversible ergodic Markov chain,
\item $$\frac{b_n}{\sum_{i=1}^{n}X_{i}^{2}}{\max_{1\le i\le n}|\sum_{j=1}^{i}(-1)^{j}X_{j}|} \to 0 \quad \mbox{as}\quad n \to \infty,$$
\item $$(\frac{b_n}{\sum_{i=1}^{n}X_{i}^{2}})^2 {\sum_{i=1}^{n}(2\sum_{j=1}^{i-1}(-1)^{j}X_{j}+(-1)^{i}X_{i})^{2}} \to 0 \quad \mbox{as}\quad n\to \infty,$$
\item $$(\frac{b_n}{\sum_{i=1}^{n}X_{i}^{2}})^2{\sum_{i=1}^{n}(\sum_{j=1}^{i}(-1)^{j}X_{j})^{2}} \to c^{2} \quad \mbox{as}\quad n\to \infty.$$
\end{enumerate}

\end{corollary}

\begin{example}
Assume that in designing the above regression model, one takes $X_i=i$. It turns out, that 
$$\tilde{d}_{n,2k}=\frac{12kb_n}{n(n+1)(2n+1)},\quad \tilde{d}_{n,2k-1}=\frac{(12k-6)b_n}{n(n+1)(2n+1)},$$ $$ \mbox{and}\quad d_{n,2k}=d_{n,2k-1}=\frac{6kb_n}{n(n+1)(2n+1)}, \quad \mbox{for }\quad k=1,\cdots, n .$$
\end{example}

It follows that $$\frac{b_n}{\sum_{i=1}^{n}X_{i}^{2}}{\max_{1\le i\le n}|\sum_{j=1}^{i}(-1)^{j}X_{j}|}\leq \frac{3b_n}{2n^2},$$
$$(\frac{b_n}{\sum_{i=1}^{n}X_{i}^{2}})^2 {\sum_{i=1}^{n}(2\sum_{j=1}^{i-1}(-1)^{j}X_{j}+(-1)^{i}X_{i})^{2}}=d_{n,1}^2+d_{n,n}^2\leq \frac{(9n^2+9)b_n^2}{n^6},$$
$$\mbox{and}\quad (\frac{b_{2n}}{\sum_{i=1}^{2n}X_{i}^{2}})^2{\sum_{i=1}^{2n}(\sum_{j=1}^{i}(-1)^{j}X_{j})^{2}} = (\frac{6b_{2n}}{2n(2n+1)(4n+1)})^2\sum_{k=1}^{n}2k^2=\frac{3b^2_{2n}(n+1)}{n(2n+1)(4n+1)^2}.$$
Therefore, if we take $b_n$ such that $\displaystyle \frac{b_n}{n^{3/2}}\to \frac{2}{3}c^2$, then all the assumption of Theorem \ref{corbeta} are satisfied and the central limit theorem (\ref{CLTbeta}) holds with $\sigma^2=2\pi h(0)c^2$.

\begin{remark}
In this example, $b_n$ is of order $n^{3/2}$. So, $b_n$ is a lot larger than $\sqrt{n}$. The sequence converges faster to its mean than under the usual CLT for i.i.d. innovations.
\end{remark}

\begin{remark}
Using $X_i=1$, we obtain $\bar{\varepsilon}$. The conditions of the theorem are not satisfied. This theorem falls short on this example, for which the result of Kipnis and Varadhan (1986) applies. 
\end{remark}

If we now consider that the sequence of explanatory variables is a random sample of variables from a distribution $f$ and is independent of the errors, then the application of Kipnis and Varadhan (1986) applies to the Reversible Markov chain $(X_i, \varepsilon_i)$ and $g(u,v)=uv$, and the following results hold.

\begin{theorem} \label{Linreg1} 
Let $(X_{i}, i\in \mathbb{N})$ be a random sample from a distribution with mean zero and finite variance $\sigma^2_x$. Let $(\varepsilon_{i}, i\in \mathbb{N})$ be a mean zero stationary ergodic reversible Markov chain with variance $\sigma^2$. Then 

\begin{equation}
\sqrt{n}(\hat{\beta}-\beta)\Longrightarrow N(0, \frac{\sigma^2}{\sigma^2_x}).
\end{equation}

\end{theorem}

\begin{theorem} \label{Linreg2} 
Let $(X_{i}, i\in \mathbb{N})$ be a random sample from a distribution with finite variance $\sigma^2_x$. Let $(\varepsilon_{i}, i\in \mathbb{N})$ be a mean zero stationary reversible ergodic Markov chain with variance $\sigma^2$. If $n^{-1}var(S_n(\varepsilon))\to \sigma^2$, then 

\begin{equation}
\sqrt{n}(\hat{\beta}-\beta)\Longrightarrow N(0, \frac{\sigma^2}{\mathbb{E}(X^2)}).
\end{equation}

\end{theorem}

The proof of Theorem \ref{Linreg1} and Theerem \ref{Linreg2} consist of simplifying $\sqrt{n}(\hat{\beta}-\beta)$ by means of the law of large numbers [$\displaystyle \frac{1}{n}\sum_{i=1}^{n}X^2\to^P \mathbb{E}(X^2)$], then computing $$\displaystyle n^{-1} var(\frac{1}{\mathbb{E}(X^2)}S_{n}(X_i\varepsilon_i))=\frac{\sigma^2 \sigma^2_x}{\mathbb{E}^2(X^2)}+\frac{\mathbb{E}^2(X)}{\mathbb{E}^2(X^2)}\frac{Var(S_n (\varepsilon))}{n}.$$ 

\begin{remark}
Theorem \ref{Linreg1} allows a mean zero reversible Markov chain $(\varepsilon_i, i\in \mathbb{N})$ without conditions on the variance of its partial sums.
This means that we can have an example of Markov chain generated by the Hoeffding Lower bound copula $W(x,y)= \max(x+y-1, 0)$ or the Hoeffding upper bound copula $M(u,v)=\min (u,v)$. These two copulas generate Markov chains that are not mixing in any sense, with variance of partial sums of order $n^2$. These cases don't apply to Theorem \ref{Linreg2}. This is the first example that involves a Markov chains with one of these two copulas.
\end{remark}

\subsubsection{The general linear regression model}

We consider now the general linear regression problem $Y_i=\alpha+\beta X_i+\varepsilon_i$, where $(\varepsilon_i, i\in \mathbb{N})$ is a stationary reversible Markov chain. The least square estimates are

$$\hat{\beta}=\frac{\sum_{i=1}^n (X_i - \bar{X})Y_i}{\sum_{j=1}^{n}(X_j - \bar{X})^2}=\beta + \sum_{i=1}^{n}\bar{\nu}_{in} \varepsilon_i, \quad\mbox{where} \quad \bar{\nu}_{in}=\frac{X_i-\bar{X}}{\sum_{j=1}^n (X_j-\bar{X})^2},$$
$$
\hat{\alpha} = \bar{Y}-\hat{\beta}\bar{X}=\alpha+ \sum_{i=1}^{n}\bar{\mu}_{in} \varepsilon_i, \quad\mbox{where} \quad \bar{\mu}_{in}=\frac{1}{n}-\frac{(X_i-\bar{X})\bar{X}}{\sum_{j=1}^n (X_j-\bar{X})^2}.
$$

Therefore, we have $$b_n (\hat{\beta}-\beta)=\sum_{i=1}^{n}\nu_{in} \varepsilon_i \quad \mbox{and}\quad b_n (\hat{\alpha}-\alpha)=\sum_{i=1}^{n}\mu_{in} \varepsilon_i , $$
where $\nu_{in}=b_n\bar{\nu}_{in}$ and $\mu_{in}=b_n \bar{\mu}_{in}.$

The central limit theorem holds for $\hat{\alpha}$ or $\hat{\beta}$ if the corresponding conditions of Theorem \ref{Rev-CLT} are satisfied. Notice that the case of equal $X$ values no longer applies here because it reduces to a one-parameter problem.
Now, if we assume that $(X_k, k\in\mathbb{N})$ is a random sample from a distribution $f$ with mean $\mu$ and variance $\sigma^2_x$, the analysis requires a different approach. Assume that the sequence of explanatory variables and the sequence of error terms are independent and $b_n=\sqrt{n}$. Then, 

$$b_n (\hat{\beta}-\beta)=\sum_{i=1}^{n}\nu_{in} \varepsilon_i = \frac{\sqrt{n}}{\sum_{j=1}^n (X_j-\bar{X})^2} \sum_{i=1}^{n}{(X_i-\mu_x)\varepsilon_i}-\frac{\sqrt{n}(\bar{X}-\mu_x)}{\sum_{j=1}^n (X_j-\bar{X})^2}\sum_{i=1}^{n}{\varepsilon_i}.$$
Using the law of large numbers, the CLT for the random sample from $f$ and simple calculations, it can be shown that 

$$\frac{\sqrt{n}(\bar{X}-\mu_x)}{\sum_{j=1}^n (X_j-\bar{X})^2}\sum_{i=1}^{n}{\varepsilon_i}\to^P 0, \quad \mbox{if} \quad var(S_n (\varepsilon))=o(n^2).$$

Thus, the limiting distribution of $\displaystyle \sqrt{n}(\hat{\beta}-\beta)$ is that of $\displaystyle \frac{1}{\sqrt{n}\sigma^2_x} \sum_{i=1}^{n}{(X_i-\mu_x)\varepsilon_i}$.
Therefore, the following can be derived by replacing variables in the proofs of Theorem \ref{Linreg1} and Theorem \ref{Linreg2}.

\begin{corollary}\label{Lincorr}
For the general linear regression model, let $(X_{i}, i\in \mathbb{N})$ be a random sample from a distribution with finite variance $\sigma^2_x$ and mean $\mu_x$. Let $(\varepsilon_{i}, i\in \mathbb{N})$ be a stationary reversible ergodic Markov chain with variance $\sigma^2$.

\begin{enumerate} 
\item If $var(S_n(\varepsilon)) =o(n^2),$ then $$\displaystyle \sqrt{n}(\hat{\beta}-\beta)\Longrightarrow N(0, \frac{\sigma^2}{\sigma^2_x}).$$

\item If $n^{-1}var(S_n(\varepsilon))\to \sigma^2$, then 

$\displaystyle \sqrt{n}
\begin{pmatrix}
\hat{\alpha}-\alpha \\
\hat{\beta}-\beta 
\end{pmatrix}
\Longrightarrow N(\begin{pmatrix}
0 \\
0 
\end{pmatrix}
, \Sigma ),
$
where 
$\displaystyle \Sigma = 
\begin{pmatrix}
\frac{\sigma^2(\sigma^2_x+\mu^2_x)}{\sigma^2_x} & -\frac{\sigma^2\mu_x}{\sigma^2_x} \\
-\frac{\sigma^2\mu_x}{\sigma^2_x} & \frac{\sigma^2}{\sigma^2_x} 
\end{pmatrix}.
$
\end{enumerate}
\end{corollary}

To prove the second statement of this corollary, we first use the law of large numbers to simplify the expression of the vector, then apply the Cramer-Wold device to the obtained result a follows. 
$$Z_n(t_1, t_2)=\sqrt{n}t_1(\hat{\alpha}-\alpha)+\sqrt{n}t_2 (\hat{\beta}-\beta)=A_n(t_1,t_2)+B_n (t_1,t_2),$$ with $B_n (t_1,t_2)\to^P 0$ under the conditions of the corollary and by Theorem \ref{Linreg2}
$$A_n(t_1,t_2)= \frac{1}{\sigma^2_x\sqrt{n}}\sum_{i=1}^{n}(t_1\sigma^2_x+t_1\mu^2_x+(t_2-t_1\mu_x)X_i-t_2\mu_x)\varepsilon_i \Longrightarrow N(0, \sigma^2_z), \quad \mbox{where}$$
$$ \sigma^2_z=\frac{(\sigma^2_x+\mu^2_x)\sigma^2}{\sigma^2}t_1^2+\frac{\sigma^2}{\sigma^2}t_2^2-2\frac{\mu_x\sigma^2}{\sigma^2}t_1 t_2= (t_1,t_2)\Sigma (t_1,t_2)'.$$
Therefore, by the Cramer-Wold device, the second statement of the corollary holds .

\begin{remark} Under the assumptions of Corollary \ref{Lincorr}, it holds that:

\begin{enumerate}
\item $var(Y)=\beta^2\sigma^2_x+\sigma^2$. $S^2_y=\frac{1}{n-1}\sum_{i=1}^{n}(Y_i-\bar{Y})^2$ is an asymptotically unbiased estimator of $var(Y)$ and $S^2_x=\frac{1}{n-1}\sum_{i=1}^{n}(X_i-\bar{X})^2$ is an unbiased estimator of $\sigma^2_x$. These quantities can be replaced in calculations by their estimates for large sample inference when they are unknown.

\item If $\mu_x=0$, then the two estimators are asymptotically independent.

\item By the Delta method, for any bivariate function $m$ such that $\nabla m \begin{pmatrix}
\alpha \\
\beta 
\end{pmatrix}\ne \bold{0},$

$\displaystyle \sqrt{n}
(m \begin{pmatrix}
\hat{\alpha} \\
\hat{\beta} 
\end{pmatrix}
-m \begin{pmatrix}
\alpha \\
\beta 
\end{pmatrix})
\Longrightarrow N(\begin{pmatrix}
0 \\
0 
\end{pmatrix}
,(\nabla m \begin{pmatrix}
\alpha \\
\beta 
\end{pmatrix})' \Sigma \nabla m \begin{pmatrix}
\alpha \\
\beta 
\end{pmatrix}) $, when $n^{-1}var(S_n(\varepsilon))\to \sigma^2.$

\end{enumerate}

\end{remark}
\subsection{The Non-parametric regression problem}
Consider the non-parametric regression problem for $g(x)=\mathbb{E}(Y|X=x)$, where $Y=g(X)+\varepsilon$ with $\mathbb{E}(\varepsilon)=0$. Assume that $K$ is a kernel function with properties that we will state later. The kernel regression estimator is the Nadaraya-Watson estimator
$$\bar{g}(x)=\frac{\sum_{i=1}^{n}K(\frac{x-X_i}{h})Y_i}{\sum_{j=1}^{n}K(\frac{x-X_j}{h})}=\frac{\sum_{i=1}^{n}K(\frac{x-X_i}{h})g(X_i)}{\sum_{j=1}^{n}K(\frac{x-X_j}{h})}+\frac{\sum_{i=1}^{n}K(\frac{x-X_i}{h})\varepsilon_i}{\sum_{j=1}^{n}K(\frac{x-X_j}{h})}$$

It follows that for a fixed sample of values of $X$, $$b_n \Big(\bar{g}(x)-\frac{\sum_{i=1}^{n}K(\frac{x-X_i}{h})g(X_i)}{\sum_{j=1}^{n}K(\frac{x-X_j}{h})}\Big) \to N(0, 2\pi h(0)c^2 ),$$ if the conditions of Theorem \ref{Rev-CLT} are satisfied with $$\tilde{d}_{n,i}=\frac{b_n K(\frac{x-X_i}{h})}{\sum_{j=1}^{n}K(\frac{x-X_j}{h})}, \quad i=1,\cdots, n.$$

\begin{theorem}\label{Lon1}
Let $(\varepsilon_k, k\in\mathbb{N})$ be a mean zero stationary reversible Markov chain independent of a random sample $(X_k, k\in\mathbb{N})$ from a distribution $f$. Assume that $f$ has a bounded first order derivative $f'$. Let $K$ be a kernel function on $\mathbb{R}$ such that $\displaystyle \int K(u)du=1$ and $\displaystyle \int K^4(u)du<\infty$. Assume that $\mathbb{E}(\varepsilon^4_i)<\infty$ and $b_n=n^{1/2-\alpha/6}$ and $h=n^{-\alpha/3}$ for some $3/5< \alpha < 3$ and $g$ has two bounded derivatives. If 
\begin{equation}
\sum_{k=1}^{n}cov(\varepsilon^2_0, \varepsilon^2_k)=o(n), 
\end{equation}

and 
\begin{equation}
\sum_{k=1}^{n}cov(\varepsilon_0, \varepsilon_k) = o(n^{\alpha/3}),
\end{equation}
then 
\begin{equation}
b_n (\bar{g}(x)-g(x))\to N(0, \frac{\sigma^2}{f(x)}\int K^2 (u)du).
\end{equation}

\end{theorem}

It can be easily shown that when $(X_{i}, i\in \mathbb{N})$ are independent random variables that are independent of $(\varepsilon_i, i\in\mathbb{N})$, under the assumptions of Theorem \ref{Lon1}. We have 

$$b_n (\bar{g}(x)-g(x))=B_n(x)+C_n(x), \mbox{\quad with \quad}$$ $$ B_n (x)=b_n \frac{\sum_{i=1}^{n}K(\frac{x-X_i}{h})(g(X_i)-g(x))}{\sum_{j=1}^{n}K(\frac{x-X_j}{h})}, \quad C_n (x)=b_n \frac{\sum_{i=1}^{n}K(\frac{x-X_i}{h})\varepsilon_i}{\sum_{j=1}^{n}K(\frac{x-X_j}{h})}, \quad \mbox{where}$$
$$Var(B_n(x))=O(\frac{hb_n^2}{n})=O(n^{-2\alpha/3}) ,\mbox{and}\quad \mathbb{E}(B_n(x))=O(b_nh^2)=O(n^{1/2-5\alpha/6})$$ $$ \mbox{and}\quad C_n (x)-\frac{b_n}{nhf(x)}{\sum_{i=1}^{n}K(\frac{x-X_i}{h})\varepsilon_i}\to^P 0 .$$
Therefore, by the Slutsky theorem, the limiting distribution of $b_n (\bar{g}(x)-g(x))$ is that of $\frac{b_n}{nhf(x)}{\sum_{i=1}^{n}K(\frac{x-X_i}{h})\varepsilon_i}.$ Thus, the proof of Theorem \ref{Lon1} reduces to that of the following.

\begin{theorem} \label{Lon2} Let $(\varepsilon_k, k\in\mathbb{N})$ be a mean zero stationary reversible Markov chain independent of a random sample $(X_k, k\in\mathbb{N})$ from a distribution $f$. Assume that $f$ has a bounded first order derivative $f'$. Let $K$ be a kernel function on $\mathbb{R}$ such that $\displaystyle \int K(u)du=1$ and $\displaystyle \int K^4(u)du<\infty$. Assume that $\mathbb{E}(\varepsilon^4_i)<\infty$ and $b_n=n^{1/2-\alpha/6}$ and $h=n^{-\alpha/3}$ for some $0\leq \alpha< 3$. If 

\begin{equation}
\sum_{k=1}^{n}cov(\varepsilon^2_0, \varepsilon^2_k)=o(n), 
\end{equation}

and 
\begin{equation}
\sum_{k=1}^{n}cov(\varepsilon_0, \varepsilon_k) = o(n^{\alpha/3}),
\end{equation}
then 
\begin{equation}
\frac{b_n}{nhf(x)}{\sum_{i=1}^{n}K(\frac{x-X_i}{h})\varepsilon_i}\to N(0, \sigma^2_g).
\end{equation}
where $\displaystyle \sigma^2_g=\frac{\sigma^2}{f(x)}\int K^2(u)du$.

\end{theorem}

We shall prove Theorem \ref{Lon2} using the following theorem from Longla and al. (2015).

\begin{theorem} \label{TheoLon}
For a stationary reversible Markov chain $(\gamma_i,i\in\mathbb{Z}),$ and for any real functions $g_n$, defining $X_{n,k} =g_n(\gamma_k)$, if 
\begin{equation}
\mathbb{E}(X_{n,k}^4)<\infty, \quad \mathbb{E}X_{n,k}=0, \quad \mathbb{E}(X_{n,0}^2)\to \sigma_g^2 \label{eq1}
\end{equation}

\begin{equation}
cov(X_{n,0}, X_{n,2})+\sum_{k=2}^{n}cov(X_{n,0}, X_{n,u}) \to 0 \label{eq2}
\end{equation}

\begin{equation}
\frac{1}{n}(var(X^2_{n,0})+\sum_{n=0}^{n}cov(X^2_{n,0}, X^2_{n,u}) \to 0, \label{eq3}
\end{equation}
then 

\begin{equation}
\frac{1}{\sqrt{n}}\sum_{k=1}^{n}X_{n,k} \to N(0, \sigma_g^2).
\end{equation}

\end{theorem}

To prove Theorem \ref{Lon2}, we shall check the assumptions of Theorem \ref{TheoLon} for the reversible Markov chain $((X_i, \varepsilon_i), i\in\mathbb{N})$ and the functions $$g_n(u,v)=\frac{b_n}{\sqrt{n}hf(x)}K(\frac{x-u}{h})v.$$
Noticing that $\mathbb{E}(\varepsilon)=0$ and $\varepsilon$ and $X$ are independent, it follows that condition (\ref{eq1}) is satisfied if $f'$ is bounded, $K\in \mathbb{L}^4(\mathbb{R})$, $\mathbb{E}(\varepsilon^4)<\infty $ and 
$$\frac{b_n^2}{nh^2f^2(x)}\mathbb{E}(K^2(\frac{x-X_1}{h}))\mathbb{E}(\varepsilon^2_1) \equiv \frac{b_n^2}{nhf(x)}\int K^2(u)du\sigma^2 \to \sigma_g^2.$$

Consider the assumption (\ref{eq2}). For any $k\ne 0$, using independence of $X_0, X_k $ and $\varepsilon$, we obtain $$cov(X_{n,0}, X_{n,k})=\frac{b_n^{2}}{nh^2f^2(x)}\Big(\mathbb{E}(K(\frac{x-X_0}{h}))\Big)^2 cov(\varepsilon_0, \varepsilon_k), \quad \mbox{and}$$ 

$$cov(X_{n,0}, X_{n,2})+\sum_{k=2}^{n}cov(X_{n,0}, X_{n,u})=\frac{b_n^{2}}{nh^2f^2(x)}\Big(\mathbb{E}(K(\frac{x-X_0}{h}))\Big)^2 (cov(\varepsilon_0, \varepsilon_2)+ \sum_{k=2}^{n}cov(\varepsilon_0, \varepsilon_k)).$$
Moreover, this quantity is equivalent to $$m_n=\frac{b_n^{2}}{n}\Big(\int K(u)du\Big)^2 (cov(\varepsilon_0, \varepsilon_2)+ \sum_{k=2}^{n}cov(\varepsilon_0, \varepsilon_k)).$$
Taking into consideration the fact that $\displaystyle \frac{b_n^2}{n}\to 0$, we can conclude that, if $$\sum_{k=1}^{n}cov(\varepsilon_0, \varepsilon_k) = o(\frac{n}{b_n^2}),$$

$$\mbox{then} \quad \lim_{n\to\infty}m_n =\lim_{n\to\infty} \frac{b_n^{2}}{n}\Big(\int K(u)du\Big)^2 (\sum_{k=1}^{n}cov(\varepsilon_0, \varepsilon_k)) =0.$$

Consider the assumption (\ref{eq3}). For any $k > 0$, using independence of $X_0, X_k $ and $\varepsilon$, we obtain
$$cov(X^2_{n,0}, X^2_{n,k})=\frac{b_n^{4}}{n^2 h^4f^4(x)}\Big(\mathbb{E}(K^2(\frac{x-X_0}{h}))\Big)^2 cov(\varepsilon^2_0, \varepsilon^2_k), \quad \mbox{and}$$ 

$$\frac{1}{n}(var(X^2_{n,0})+\sum_{n=0}^{n}cov(X^2_{n,0}, X^2_{n,k}))=\frac{b_n^{4}\Big(\mathbb{E}(K^2(\frac{x-X_0}{h}))\Big)^2 }{n^3 h^4f^4(x)} (\sum_{k=1}^{n}cov(\varepsilon^2_0, \varepsilon^2_u)) + \frac{2b_n^4 \mathbb{E}K^4(\frac{x-X_0}{h})var(\varepsilon_0^2))}{n^3h^4f^4(x)}.$$
It is clear that this last quantity has the same limit as $n\to \infty$ as 
$$M_{1n}=\frac{b_n^{4}}{n^3 h^2f^2(x)}(\int K^2(u))^2du (\sum_{k=1}^{n}cov(\varepsilon^2_0, \varepsilon^2_k)) + \frac{2b_n^4 var(\varepsilon_0^2))}{n^3h^3f^3(x)}\int K^4(u)du .$$

It is obvious, that $\displaystyle \lim_{n\to\infty} M_{1n}=0$, iff $\displaystyle \frac{b_n^4}{n^3h^3}\to 0$ and $\displaystyle \sum_{k=1}^{n}cov(\varepsilon^2_0, \varepsilon^2_k)=o(\frac{n^3h^2}{b_n^4})$.
This ends the proof of Theorem \ref{Lon2}. To complete the proof of Theorem \ref{Lon1}, it remains to notice that the equivalence of the two limiting distributions required $\alpha >3/5.$

\end{document}